\documentclass{amsart}

\usepackage{mathrsfs}
\usepackage{mathabx}
\usepackage{mathtools}

\usepackage{hyperref}
\usepackage{fancyhdr, amsthm, textcomp}
\usepackage{latexsym}

\usepackage[utf8]{inputenc}
\usepackage{amsfonts,amsmath,amscd,amssymb}
\usepackage{xspace}

\usepackage{caption}
\usepackage{subcaption}
\usepackage{graphicx}

\DeclareMathAlphabet{\mathcalligra}{T1}{calligra}{m}{n}

\DeclareMathOperator{\id}{id}

\DeclareMathOperator{\cohdim}{coh.dim}

\DeclareMathOperator{\End}{End}
\DeclareMathOperator{\Ad}{Ad}
\DeclareMathOperator{\ad}{ad}
\DeclareMathOperator{\GL}{GL}
\DeclareMathOperator{\SL}{SL}

\DeclareMathOperator{\stab}{Stab}
\newtheorem{thm}{Theorem}[section]

\newtheorem{prop}[thm]{Proposition}
\newtheorem{cor}[thm]{Corollary}

\theoremstyle{definition}

\newcommand{\ve}{\ensuremath{\mathbf{e}}\xspace}

\newcommand{\vs}{\ensuremath{\mathbf{s}}\xspace}
\newcommand{\vt}{\ensuremath{\mathbf{t}}\xspace}

\newcommand{\sC}{{\mathscr C}}
\newcommand{\sF}{{\mathscr F}}
\newcommand{\sE}{{\mathscr E}}
\newcommand{\sH}{{\mathscr H}}
\newcommand{\sL}{{\mathscr L}}
\newcommand{\sQ}{{\mathscr Q}}

\newcommand{\mA}{{\mathcal A}}

\newcommand{\mC}{{\mathcal C}}
\newcommand{\mD}{{\mathcal D}}
\newcommand{\mE}{{\mathcal E}}

\newcommand{\mG}{{\mathcal G}}

\newcommand{\mM}{{\mathcal M}}
\newcommand{\mP}{{\mathcal P}}
\newcommand{\mR}{{\mathcal R}}
\newcommand{\mS}{{\mathcal S}}
\newcommand{\mT}{{\mathcal T}}
\newcommand{\mU}{{\mathcal U}}
\newcommand{\mX}{{\mathcal X}}

\newcommand{\bA}{{\mathbb A}}
\newcommand{\bC}{{\mathbb C}}
\newcommand{\bF}{{\mathbb F}}
\newcommand{\bK}{{\mathbb K}}
\newcommand{\bN}{{\mathbb N}}
\newcommand{\bQ}{{\mathbb Q}}
\newcommand{\bR}{{\mathbb R}}

\newcommand{\bZ}{{\mathbb Z}}

\newcommand{\fB}{{\mathfrak B}}

\newcommand{\fD}{{\mathfrak D}}

\newcommand{\fX}{{\mathfrak X}}

\newcommand{\fg}{{\mathfrak g}}
\newcommand{\fn}{{\mathfrak n}}
\newcommand{\fh}{{\mathfrak h}}

\newcommand{\Vt}{\ensuremath\underline{\underline{V}}}
\newcommand{\Al}{\ensuremath\underline{\underline{A_\Lambda}}}
\newcommand{\Ch}{{\mathrm C}{\mathrm s}{\mathrm h}_{G^\star}(\fD)}

\begin{document}
\title{Kac-Moody Groups and Their Representations}
\author{Dmitriy Rumynin}
\email{D.Rumynin@warwick.ac.uk}
\address{Department of Mathematics, University of Warwick, Coventry, CV4 7AL, UK
  \newline
\hspace*{0.31cm}  Associated member of Laboratory of Algebraic Geometry, National
Research University Higher School of Economics, Russia}
\thanks{The research was partially supported by the Russian Academic Excellence Project `5--100' and by Leverhulme Foundation. The author would like to thank Inna Capdeboscq and 
Timoth\'{e}e Marquis
for valuable discussions.}
\date{December 18, 2017}

\dedicatory{To Leonid Arkadievich Bokut with admiration}
\begin{abstract}
  In this expository paper we review some recent results about
  representations of Kac-Moody groups. We sketch the construction of these groups. 
  If practical, we present the ideas behind the proofs of theorems.
  At the end we pose open questions.
\end{abstract}

\maketitle

Kac-Moody Lie algebras are well-known generalisations
of simple finite-dimensional Lie algebras,
subject of 1533 research papers on MathSciNet and
at least 3 beautiful monographs \cite{Kac,Wan,Carter}
Kac-Moody groups are less well-known cousins,
subject of only 214 research papers on MathSciNet.
One issue with them is that there are several different notions
of a Kac-Moody group:
\begin{itemize}
\item a group valued functor on commutative rings defined by Tits,
a generalisation of  $R \mapsto \SL_n (R[z,z^{-1}])$,
\item a locally compact totally disconnected group,
  a generalisation of $\SL_n (\bF_q ((z)) \; )$,
\item an ind-algebraic group,
  a generalisation of $\SL_n (\bC ((z)) \; )$,
\item  a more complicated topological group, e.g.,
$\SL_n (\bQ_p ((z)) \; )$.
\end{itemize}
In this survey, we review some new results
about the first two types of Kac-Moody groups
and their representations. We give examples and sketch proofs
whenever it is practical. The only completely new results
are in Section~\ref{s1.4} where full proofs are given.

There are instructional sources about their Group Theory
and Geometry \cite{CaprRem, Marquis, Remy}
but not about their Representation Theory.
A reader interested in ind-algebraic Kac-Moody groups
can consult a monograph \cite{Ku}
but someone who wants to learn about more complicated Kac-Moody groups
will need to look at scholarly sources \cite{Rou1,Rou2}. 
We start without further ado.

\section{Representations of uncompleted group} \label{s1}
\subsection{Kac-Moody Lie algebra}\label{s1.1}
Let $\mA=(A_{i,j})_{n\times n}$ be a square
matrix with coefficients in a commutative
ring $\bK$.
{\em A realisation of $\mA$} is a collection 
$\mR = (\fh,h_1,\ldots h_n,\alpha_1, \ldots \alpha_n)$
where $\fh$ is a finitely generated free $\bK$-module,
$h_i$ are $\bK$-linearly independent elements of $\fh$,
$\alpha_j$ are $\bK$-linearly independent elements of $\fh^\ast$,
and
$\alpha_j (h_i) = A_{ij}$ for all $i$ and $j$.

A realisation gives several interesting Lie $\bA$-algebras
for any commutative $\bK$-algebra $\bA$. 
The first Lie $\bA$-algebra is $\widetilde{L}_\mR(\bA)$:
it is generated by $\fh_\bA = \fh\otimes_{\bK}\bA$ and elements
$e_1, \ldots e_n, f_1 \ldots f_n$ subject to the relations
$$
[h,e_i] = \alpha_i (h) e_i,\;  [h, f_j] = -\alpha_j(h) f_j, \;
[h^\prime ,h]=0, \ 
[e_i, f_i] = h_i, \ [e_i, f_j] = 0 \mbox{ if } i\neq j
$$
for all $h^\prime ,h\in \fh_\bA$. 
The Lie algebra $\widetilde{L}_\mR(\bA)$ is graded by the root group
$X(\mR)$, the free abelian group generated by elements $\alpha_i$.
The grading  is given by
$$
\deg (h) = 0, \; \deg (e_i) = \alpha_i, \; \deg (f_i) = -\alpha_i \; .
$$
Let $I_\bA$ be the sum of all ideals of $\widetilde{L}_\mR(\bA)$,
contained in the non-zero graded part
$\oplus_{\gamma\neq 0} \widetilde{L}_\mR(\bA)_\gamma$.
The second Lie algebra is
$\widehat{L}_\mR(\bA)\coloneqq \widetilde{L}_\mR(\bA)/I_\bA$
and the third Lie algebra is
%
${L}_\mR(\bA)\coloneqq \widehat{L}_\mR(\bK)\otimes_\bK \bA$.
Although there is some literature on ${L}_\mR(\bA)$
for a general $\mA$ \cite{VK},
these algebras merit further investigation
(cf.~\ref{s3.1} and \ref{s3.2}).

If $\mA$ is a generalised Cartan matrix,
we set $\bK=\bZ$,
call a realisation (over $\bZ$)
{\em a root datum}
and denote it $\mD$.
While both
$\widehat{L}_\mD(\bA)$
and
${L}_\mD(\bA)$
deserve to be called {\em Kac-Moody algebras},
the actual definition of a Kac-Moody algebra is different.
Let $\mU_\bZ$ be the divided powers integral form
of the universal enveloping algebra
$U({L}_\mD(\bC))$. Then {\em a Kac-Moody algebra} is
defined as
$$
\fg_\bZ \coloneqq {L}_\mD(\bC) \cap \mU_\bZ \; , \ \ 
\fg_\bA \coloneqq \fg_\bZ \otimes_\bZ \bA \; . 
$$
It inherits a triangular decomposition
$\fg_\bA =
(\fn_{-} \otimes \bA)
\oplus
(\fh \otimes \bA)
\oplus
(\fn_{+} \otimes \bA)$
from
$\fg_\bZ =
\fn_{-}
\oplus \fh \oplus
\fn_{+}$
where
$\fn_{- \, \bC}$ 
is the Lie subalgebra of
$\fg_{\bC}$ 
generated by all $f_i$,
$\mU_-$ is the divided powers $\bZ$-form of $U(\fn_{-\, \bC})$
and 
$\fn_{-}\coloneqq \fn_{-\, \bC}\cap \mU_-$
(ditto for $\fn_+$ using $e_i$-s and $\mU_+$).
If $\bF$ is a field of characteristic $p$,
the Lie algebra $\fg_\bF$
is restricted with the $p$-operation
$$
(h\otimes 1)^{[p]} = h\otimes 1, \ \ 
(x\otimes 1)^{[p]} = x^p \otimes 1
\ \mbox{ where } \
h\in \fh, \; x\in \fn_\pm
$$
where $x^p$ is calculated inside the associative $\bZ$-algebra
$\mU_\pm \leq \mU$
\cite[Th. 4.39]{Marquis}. 

If $p>\max_{i\neq j} (-A_{i,j})$, then all the three Kac-Moody
coincide: $\widehat{L}_\mD(\bF) = {L}_\mD(\bF) = \fg_\bF$
\cite{Rou2, Marquis}
but it is probably no longer true for small primes.



\subsection{Kac-Moody group}\label{s1.2}
The algebras $\mU_\bZ$ and $\fg_\bZ$ inherit the grading by $X(\bR)$.
It is also known as
the root decomposition
$$
\fg_\bC = \bigoplus_{\alpha\in \Phi \subseteq X(\mD)} \fg_{\bC\, \alpha}
\; .
$$
The set of roots splits into two disjoint parts:
real roots $\Phi^{re} \coloneqq W\{\alpha_1, \ldots \alpha_n\}$
(where $W$ is the Weyl group)
and imaginary roots $\Phi^{im} \coloneqq \Phi \setminus \Phi^{im}$.

{\em The Kac-Moody group} is a functor $G_\mD$ from commutative rings to groups.
Its value on a field $\bF$ can be described as
$$
G_\mD (\bF)
\; = \;
T \ast \Asterisk_{\alpha\in\Phi^{re}} U_\alpha
/ \langle \;\mbox{Tits' relations}\;\rangle
, \ \
T= \fh\otimes_\bZ \bF^\times, \ \
U_\alpha \cong  \bF^+,
$$
where $T$ is a torus
and 
$U_\alpha = \{ X_\alpha (\vt) \}$, 
$X_\alpha (\vt)X_\alpha (\vs)=X_\alpha (\vt+\vs)$
is a root subgroup. 
There are different ways to write Tits' relations: the reader should consult
classical papers \cite{CaCh,Tits}
for succinct presentations. 
Note that Tits' Relations have infinitely many generators and relations
unless $\mA$ is of finite type. 

However, if the field $\bF=\bF_q$, $q=p^m$ is finite
and under mild assumptions on $\mA$, the groups $G_\mD (\bF_q)$
are finitely presented \cite{AbMu}
(cf. \cite{CKR2} for concrete finite presentations  of affine groups)
and simple \cite{CaprRem1}. 
Thus, the groups $G_\mD (\bF_q)$
form a good source of
finitely-presented simple (non-linear) groups. 

It is important for us that they have a BN-pair with
$B= T \ltimes U_+$ where $U_+$ is the subgroup generated by
all $U_\alpha$ for positive real roots $\alpha$.

\subsection{Adjoint representation}\label{s1.3}
The group $G_\mD (\bF)$ acts the Lie algebra $\fg_\bF$ via adjoint
action \cite{Marquis,Remy}.
The torus action comes from the $X(\mD)$-grading
$$
\Ad (h\otimes \vt) ( a )
=
\vt^{\alpha (h)} a
\ \mbox{ where } \
a\in (\fg_\bF)_\alpha
\; 
$$
and the action of $U_\alpha$ is exponential:
$$
\Ad (X_\alpha (\vt)) (a)
= \ve^{\ad (\vt e_\alpha)} (a)
=
\sum_{n=0}^\infty \vt^n \ad (e_\alpha^{(n)}) (a)
$$
where $e_\alpha$ (rather than $e_\alpha \otimes 1$)
is a non-zero element of $\fg_{\bF\, \alpha}$ and
$e_\alpha^{(n)}\in \mU\otimes \bF$ is its divided power.
Notice that $\fg_{\bF\, \alpha}$ is one-dimensional
for a real root $\alpha$.

We denote the image of $\Ad$ by $G^{ad}_\mD (\bF)$.

\subsection{Over-restricted representations}\label{s1.4}
Let $\bF$ be a field of positive characteristic $p$
in this section.
A representation $(V,\rho)$ of the Lie algebra $\fg_\bF$
is called {\em restricted}
if $\rho (x)^p = \rho (x^{[p]})$ for all $x\in\fg_\bF$.
Each real root $\alpha$ yields an additive family
of linear operators on a restricted representation
$$
Y_\alpha (\vt) \coloneqq \ve^{\rho (e_\alpha)}
= \sum_{k=0}^{p-1} \frac{1}{k!}\rho (e_\alpha)^k.
$$
These operators do not define an action of $G_\mD (\bF)$
in general. 
The concept of {\em an over-restricted} representation,
proposed recently
to integrate representations
from Lie algebras to algebraic groups \cite{RW},
proves beneficial here as well.
We say that a restricted representation
$(V,\rho)$ of  $\fg_\bF$
is {\em over-restricted}
if $\rho (e_{\alpha})^{\lfloor (p+1)/2\rfloor}=0$
for any real root $\alpha$. 

\begin{prop} (cf. \cite{RW})
\label{abs_chev}
Suppose that
$(V,\rho)$ is an
over-restricted representation
$\fg_\bF$.
If
$\ad (e_{\alpha}^{(p)}) (x)=0$ for some $x\in\fg_\bF$,
then
\begin{equation}
  \label{eq1}
\rho\big( \Ad (X_{\alpha}(\vt))(x) \big)
=
Y_\alpha (\vt)
\rho (x)
Y_\alpha (-\vt) \; . 
\end{equation}
\end{prop}
\begin{proof}
Observe by induction that for each $k=1,2,\ldots p-1$
\begin{equation}
  \label{ind_f}
\rho(\frac{1}{k!}\ad(e_\alpha)^k(x)) =
\sum_{j=0}^k \frac{(-1)^j}{(k-j)!j!} \rho(e_\alpha)^{k-j}\rho(x)\rho(e_\alpha)^j
\;
.
\end{equation}
The condition $\ad (e_{\alpha}^{(p)}) (x)=0$
implies that $\ad (e_{\alpha}^{(n)}) (x)=0$
for all $n\geq p$ and
$\rho\big( \Ad (X_{\alpha}(\vt))(x) \big)
=
\sum_{k=0}^{p-1} \rho(\frac{1}{k!}\ad(\vt e_\alpha)^k(x))
$
stops at degree $p-1$.
Using Formula~(\ref{ind_f}), this is equal to
$$
\sum_{i+j=0}^{p-1}
\frac{(-1)^j}{i!j!} \rho(\vt e_\alpha)^{i}\rho(x)\rho(\vt e_\alpha)^j
\stackrel{\spadesuit}{=}
\sum_{i,j=0}^{p-1} \frac{(-1)^j}{i!j!} \rho(\vt e_\alpha)^{i}\rho(x)\rho(\vt e_\alpha)^j
=
\ve^{\rho(\vt e_\alpha)}\rho (x) \ve^{-\rho(\vt e_\alpha)},
$$
exactly the right hand side.
Notice that
Equality~{$\spadesuit$} holds because $(V,\rho)$ is over-restricted:
terms on the right,
missing from the left, are all zero.
\end{proof}

Consider an $X(\mD)$-graded restricted representation $(V,\rho)$ of $\fg$.
Grading gives an action of $T$ on $V$ by
$\widehat{\rho} (h\otimes \vt) ( v_\alpha )
=
\vt^{\alpha (h)} v_\alpha$.
An analogue of Proposition~\ref{abs_chev}
holds for $T$:
\begin{equation}
  \label{eq2}
\rho\big( \Ad (h \otimes \vt) (x) \big)
=
\widehat{\rho} (h \otimes \vt)
\rho (x)
\widehat{\rho} (h \otimes \vt^{-1})
\; . 
\end{equation}
Let $G_V$ be the subgroup of $\GL (V)$ generated by $\widehat{\rho}(T)$ and all $Y_\alpha (\vt)$.
\begin{thm} \label{extension_g}
  Suppose that $p>\max_{i\neq j}(-A_{i,j})$.
If $(V,\rho)$ is an $X(\mD)$-graded
over-restricted representations of $\fg_\bF$,
faithful on both $T$ and $\fg_\bF$,
then
$$
\phi: G_V\rightarrow G^{ad}_\mD (\bF), \ 
\phi (Y_\alpha (\vs) ) = X_\alpha (\vs), \ 
\phi (\widehat{\rho}(t)) = t
\ \mbox{ for } \
t\in T
$$
is a surjective homomorphism of groups
whose kernel is central and consists of $\fg_\bF$-automorphisms
of $V$.
\end{thm}
\begin{proof}
  Let 
  $H$ be the free product of $T$ and all additive groups
  $U_\alpha$, $\alpha\in\Phi^{re}$.
  Both $G_V$ and $G^{ad}_\mD(\bF)$ are naturally quotients of $H$.
  If $x_1\ast\ldots \ast x_n \in \ker (H\rightarrow G_V)$
  where all $x_i$ are from the constituent groups then
$$
\phi (x_1)\phi (x_2) \ldots \phi (x_n) = I_V
\; .
$$
Formulas~(\ref{eq1}) and (\ref{eq2})
imply
that 
$$
\rho \big( [\Ad (x_1)\Ad (x_2) \ldots \Ad (x_n)] (e_\beta) \big)
= 
\rho (e_\beta), \ \
\rho \big( [\Ad (x_1)\Ad (x_2) \ldots \Ad (x_n)] (h) \big)
= 
\rho (h)
$$
for all $h\in \fh_\bF$ and real roots $\beta$.
Our restriction on $p$ imply that $\fg_\bF$ is generated
by $\fh_\bF$ and all $e_\beta$ \cite{Rou2}.
Consequently, 
$$
\rho \big( [\Ad (x_1)\Ad (x_2) \ldots \Ad (x_n)] (y) \big)
= 
\rho (y)
$$
for all $y\in \fg_\bF$.
Since $\rho$ is injective it follows that
$[\Ad (x_1)\Ad (x_2) \ldots \Ad (x_n)] (y) = I_\fg$
  and
  $x_1\ast\ldots \ast x_n \in \ker (H\rightarrow G^{ad}_\mD (\bF))$.
Hence, 
$\phi$ is well-defined.

It remains to determine the kernel of $\phi$.
Suppose
$y = x_1 x_2 \ldots x_n \in \ker (\phi)$
with all $x_i$ are either $Y_{\alpha} (\vs)$,  or in $T$.
Arguing as above, 
$\rho(z) = \rho (\phi (y) (z)) = y\rho (z) y^{-1}$
for all $z\in\fg$.
So $y\in \End (V,\rho)$:
it commutes with all $\rho(e_\alpha)$, hence with all $Y_\alpha (\vs)$.
Since $T$ acts faithfully, $y$ commutes
with $T$ as well. Commuting with all generators of $G_V$,
$y$ is inevitably central.
\end{proof}

As soon as there are few endomorphisms, the map $\phi$ in Theorem~\ref{extension_g}
can be ``reversed'' to define a projective representation of the Kac-Moody group.

\begin{cor}
  Suppose that in the conditions of Theorem~\ref{extension_g}
  the representation 
  $(V,\rho)$ is a brick, i.e.,
  $\End (V,\rho) = \bF$.
Then
$$
\theta: G_\mD (\bF) \rightarrow \GL (V), \
\theta (X_\alpha (\vs) ) = Y_\alpha (\vs), \ 
\theta (t) = \widehat{\rho}(t)
\ \mbox{ for } \
t\in T
$$
extends to a projective representation of $G_\mD (\bF)$.
\end{cor}

If the root datum is simply-connected, i.e., $\alpha_i$ form a basis of $\fh$,
then the group $G_\mD (\bF)$ is generated by $U_\alpha$-s \cite{CaCh}.
Hence, no grading is needed to define a representation of  $G_\mD (\bF)$,
with all the proofs going through as before: 
\begin{cor}
  Suppose that $\mD$ is simply-connected and
  $p>\max_{i\neq j}(-A_{i,j})$.
If $(V,\rho)$ is a faithful, over-restricted brick for $\fg_\bF$,
then
$$
\theta: G_\mD (\bF) \rightarrow \GL (V),  \
\theta (X_\alpha (\vs) ) = Y_\alpha (\vs)
$$
extends to a projective representation of $G_\mD (\bF)$.
\end{cor}

\section{Representations of completed group} \label{s2}
\subsection{Completion}\label{s2.1}
The group $G_{\mD} (\bF)$
is also known as the ``minimal'' Kac-Moody group,
while some of its various completions
$\widehat{G_{\mD} (\bF)}$
go under the name a ``maximal'' Kac-Moody group.

Let us consider a group $G$ with a BN-pair $(B,N)$.
Let $\widehat{G}$ be a completion of $G$ with respect to
some topology. 
Is $(\overline{B},N)$
(where $\overline{B}$ is the closure of $B$ in $\widehat{G}$)
a  BN-pair on $\widehat{G}$? 
It depends on circumstances. For example,
consider 
a simple split group scheme $\mG$, $G=\mG (\bF_q [z,z^{-1}])$, 
the group of monomial matrices $N\leq G$
and
positive and negative Iwahori subgroups
$I_\pm = [\mG(\bF_q [z^{\pm 1}])\xrightarrow{z^{\pm 1}\mapsto 0} \mG (\bF_q)]^{-1} (B)$.
Both pairs $(I_\pm ,N)$ are BN-pairs on $G$ but
only $(\overline{I_+},N)$ is a BN-pair on the positive completion
$\widehat{G}=\mG (\bF_q((z))\,)$:
the countable groups $\widehat{I_-}=I_-$ and $N$ cannot generate uncountable $\widehat{G}$.
The following theorem pinpoints the completion process
for groups with a BN-pair under some conditions:
\begin{thm} \cite[Th. 1.2]{CR}
  \label{complete_2}
Let $G$ be a group with a BN-pair $(B,N)$
with Weyl group $(W,S)$ where $S$ is finite. 
Suppose further that a topology $\mT$ on $B$ is given
such that
the four conditions ($1$)--($4$)
hold.
\begin{enumerate}
\item $(B,\mT)$ is a topological group.
\item The completion $\widehat{B}$ is a group.
\item $\mT_1 \coloneqq \{ A\in \mT \,\mid\, 1\in A\}$ is a basis at $1$ of topology on each minimal parabolic
  $P_s$, $s\in S$ that defines a structure of topological group on $P_s$.
\item The index $|P_s:B|$ is finite for each $s\in S$. 
\end{enumerate}
Under these conditions the following statements hold:  
\begin{enumerate}
\item[(a)]
  $\mT_1$ is a basis at $1$ of topology on $G$
  that defines a structure of topological group on $G$.
\item[(b)]  The completion $\widehat{G}$ is a group.
  The completion $\widehat{B}$ is equal to the closure $\overline{B}$.
\item[(c)] The completion $\widehat{G}$ is isomorphic
  to the amalgam $\underset{\fB}{\ast}\;H$
  where \newline 
  $\fB=\{\overline{B},N,\overline{P_s}; s\in S\}$.
\item[(d)] The pair
$(\overline{B},N)$ is a BN-pair on the completed group $\widehat{G}$.
\end{enumerate}
\end{thm}
\begin{proof}
  It is a well-known theorem of Tits
  that a group with a BN-pair is an amalgam
  of $N$ and its minimal parabolics \cite[Prop. 5.1.7]{Ku}.
Later on Tits has shown how 
to back-engineer this group from such amalgam \cite{Tits1} 
(cf. \cite[Th. 5.1.8]{Ku}). 

This is the heart of the proof:
we pinpoint the completed group in part (c)
but need to check numerous technical conditions of the Tits theorem.
See \cite{CR} for full details.
\end{proof}

Theorem~\ref{complete_2}
gives us {\em a locally pro-$p$-complete} Kac-Moody group
$G^{lpp}\coloneqq \widehat{G_\mD (\bF_q)}$, $q=p^m$ 
  by choosing the pro-$p$-topology on $B$:
  its basis at $1$ is 
  $\{ A\leq B\mid  |B:A|=p^a\ \mbox{for some}\ a\in\bN\}$. 
  The Borel and the minimal parabolic subgroups of $G_\mD (\bF_p)$ are split
  \cite[6.2]{CaprRem} :
  $$
  B=T \ltimes U_+, \ \ 
  P_s= L_s\ltimes U_s
  \mbox{ where }
  L_s=\langle U_{\alpha_i}\cup U_{-\alpha_i}\rangle \, T , \ 
  s=s_i, \  
  U_s\coloneqq U_+ \cap s U_+ s^{-1}.
  $$
In particular,
$|P_i:B|$ is finite for all $i\in I$
so that,
by Theorem~\ref{complete_2}, 
we can complete 
$G_\mD (\bF_q)$
with respect to the pro-$p$-topology on $B$
(or, in fact, any ``$M_s$-equivariant'' topology). 
The group $G^{lpp}$ has a BN-pair $(\widehat{B}, N)$
where $\widehat{B}=H\ltimes \widehat{U_+}$ and 
$\widehat{U_+}$ is the full pro-$p$ completion of $U$.

The {\em congruence subgroup}
$C(G^{lpp})=\cap_{g\in G^{lpp}} \widehat{U_+}^g$  
is of crucial interest. 
Suppose that $A$ is irreducible and the root datum $\mD$  is simply connected.
Let $Z'(G^{lpp})\coloneqq
Z((G_\mD (\bF_q))\times C(G^{lpp})$ (note that the intersection is trivial). 
\begin{thm} \cite{CR,Mar1,CER}
Under these conditions
$G^{lpp}/Z'(G^{lpp})$ is a topologically simple group.
Moreover, 
if $A$ is $2$-spherical, then
$G^{lpp}/Z'(G^{lpp})$ 
is an abstractly simple group.
\end{thm}
  
\subsection{Comparison to other completions}\label{s2.2}
It is instructive to compare $G^{lpp}$
with other completions of $G_\mD (\bF_q)$, a.k.a
topological Kac-Moody groups (cf. \cite{CR, Marquis,Rou2,CER,RWe}).
Let us list them:
\begin{itemize}
\item
{\em the Caprace-R\'emy-Ronan group} $G^{crr}$, a completion in the topology
of the action on Bruhat-Tits building,
\item
  {\em the Carbone-Garland group} $G^{c\lambda}$, a completion in the topology
  of the action on the integrable simple module with a highest weight $\lambda$,
\item {\em the Mathieu-Rousseau group}  $G^{ma+}$,
  an analogue of the ind-algebraic completion,
also  obtained as an amalgam $\underset{\fB}{\ast}\;H$
  where $\fB=\{\widetilde{B},N,\widetilde{P_s}\}$
  with specially constructed groups $\widetilde{P_s}$.
\item another {\em Mathieu-Rousseau group}  $G^{+}$,
  the closure of $G$ in $G^{ma+}$,
\item {\em the Belyaev group}  $G^{b}$,
  the ``largest'' completion with 
  compact totally disconnected $\overline{U_+}$.
  \item {\em the Schlichting group}  $G^{s}$,
    the ``smallest'' completion with 
    compact totally disconnected $\overline{U_+}$.
\end{itemize}

If $p>\max_{i\neq j} (-A_{ij})$, then
${G}^+ = G^{ma+}$ but they could be different, in general
\cite[6.11]{Rou2}.
The precise meaning of the ``largest'' and the ``smallest'' of the last two groups
is a certain universal property (consult \cite{RWe} for precise statement).
The action on the Bruhat-Tits building ensures that $G^s=G^{crr}$.
Theorem~\ref{complete_2} gives $G^{b}$
by considering the profinite topology on $B$ instead
of the pro-$p$-topology:
its basis at $1$ is $\{ A\leq B\mid  |B:A|< \infty \}$.
The following theorem compares the known completions:
\begin{thm} \cite{CR, Rou2, RWe}
There are  open continuous surjective group homomorphisms:
$$
G^{b}\twoheadrightarrow
G^{lpp}\twoheadrightarrow
{G}^+
\twoheadrightarrow G^{c\lambda}\twoheadrightarrow G^{crr}
\xrightarrow{\cong} G^{s}
.$$
  \end{thm}

\subsection{Davis Building}\label{s2.3}
Let $G^\star$ be one of the locally compact,
totally disconnected 
groups from Section~\ref{s2.2}.
It admits a BN-pair $(B^\star , N)$
with the same Weyl group
$(W,S)$ as the Kac-Moody algebra $\fg_\bC$. 
Consequently, $G^\star$ acts on two simplicial complexes:
the Bruhat-Tits building $\fB$ and {\em the Davis building} $\fD$
(also known as {the Davis realisation} \cite[Section 12.4]{AbBr}
or {the geometric realisation} \cite[Section 18.2]{Davis1}). 
Notice that there are variations in this definition:
the original Davis' definition produces a cell complex,
while $\fD$ (defined below) is a simplicial complex, a subdivision
of this cell complex.

While the building $\fB$ is well-known, it is still instructive
to recall its definition. Let $\mP (G^\star)$
be the set of all proper parabolic subgroups of $G^\star$.
A parabolic $P\in \mP (G^\star)$
is conjugate to precisely one of the standard parabolics
$P_J \coloneqq \langle B^\star, \dot{s}\rangle_{s\in J}$
where $J\subset S$, $\dot{s} \in G^\star$ is a lift
of the element $s\in W = N(T)/T$.
Thus, we can define {\em the type} and {\em the rank} of each parabolic by
$$
t(P) = J, \ r(P) = |J|
\ \mbox{ whenever } \ 
P \sim P_J.
$$

The building $\fB$ is an $n$-dimensional simplicial complex
($n=|S|$) 
whose set of $k$-dimensional
simplices $\fB_k$ is equal to $t^{-1} (n-k) = \{ P \,\mid\, r(P)= n-k \}$.
A simplex $P^\prime$ is a face of $P$ if and only if $P\subseteq P^\prime$.
The group $G^\star$ acts on $\fB$ in the obvious way: $\,^gP=gPg^{-1}$.
Since parabolic subgroups are self-normalising, the stabiliser of $P$ is $P$
itself. One drawback of this action is that stabilisers of simplices
are not necessarily compact. This drawback is fixed in the Davis building.

A subset $J \subset S$ is called \emph{spherical}
if the Coxeter subgroup $\langle J\rangle$ is finite.
If $\mA$ has no components of finite type, 
$\mP^{sp} (G^\star)$ is the subset of
$\mP (G^\star)$
that consists of parabolics of spherical type.
In general, $G^\star = G_1^\star \times \ldots \times G_k^\star$
with $G_j^\star$ corresponding to connected components of $\mA$. 
The elements of $\mP^{sp} (G^\star)$ are {\em marked parabolics}
$H\leq P$ where $P=P_1 \times \ldots \times P_k$ is a finite type parabolic,
$P_j\leq G_j$,
$H=H_1 \times \ldots \times H_k$
and each $H_j$ is either $\{1\}$ (if $P_j\neq G_j$),
or a Borel subgroup (if $P_j= G_j$).

The set $\mP^{sp} (G^\star)$ is partially ordered:
$(H'\leq P')\preceq (H\leq P)$ if and only if $H\subseteq P'\subseteq P$. 
The Davis building $\fD$ is 
the geometric realisation of the poset $\mP^{sp} (G^\star)$, i.e.,
its set $\fD_k$ of $k$-dimensional
simplices consists of $(k+1)$-long chains of marked spherical parabolics
$$
P_0 \prec P_1 \prec \; \ldots \; \prec P_{k-1} \prec P_k \; .
$$
Faces of a simplex are its subchains. 
The group $G^\star$ acts on $\fD$ in the same obvious way:
$g(H_i\leq P_i)=(gH_ig^{-1}\leq gP_ig^{-1})$.
The stabiliser of a chain $(P_i)$ is an open subgroup $P'_0$ of $P_0$.
Thus, all stabilisers are compact because
spherical parabolic subgroups  
are necessarily compact.

One interesting example is {\em a generic} Kac-Moody group.
Suppose $A_{i,j}A_{j,i}\geq 4$ for all $i$, $j$.
Then the only spherical subsets of $S$ are the empty set and one element
subsets. Consequently, any chain in $\mP^{sp} (G^\star)$
is of length at most 1 and
$\fD$ is a tree.

If $\mA$ has no irreducible components of finite type, both buildings
$\fB$ and $\fD$ are contractible. 
If $\mA$ has an irreducible component of finite type,
then
$\fD$ is still contractible, while $\fB$ is not
(see \cite{Davis1} for this as well as detailed study of $\fD$).
We finish this section by stating Davis' Theorem:
\begin{thm} \cite{Davis2}
  $\fD$ admits a locally Euclidean, $G^\star$-invariant metric
  that turns $\fD$ into 
  complete, CAT(0) geodesic space.
\end{thm}

\subsection{Projective Dimension of Smooth Representations}\label{s2.4}
We study representations 
of $G^\star$ over
a field $\bK$ of characteristic zero. 
A representation 
$(V,\rho)$ is called \textit{a smooth representation}
if  for all $v \in V$ there exists
a compact open subgroup $K_v\leq G^\star$ 
such that $\rho(g)v=v$ for all $g \in K_v$.
Equivalently, the action $G\times V\rightarrow V$ is required
to be continuous with respect to the discrete topology
on $V$ (and standard topologies in $G^\star$ and the product).

The category $\mM(H)$ of smooth representations
of a locally compact totally disconnected topological
group $H$ 
is abelian with enough projectives \cite{Bern,LL}.
In case of the group $G^\star$ we can say more
by examining its action on $\fD$:
\begin{thm} \cite{HR}
\label{PD_dim}
  Let $d$ be the dimension of $\fD$. Then
$$
\cohdim(\mM(G^\star)) \leq d.
$$
\end{thm}
\begin{proof}
  Let $C_i = C_i (\fD, \bK)$ be the group of $\bK$-linear
  chains on $\fD$.
The chain complex 
$
\mathscr{C} = (C_{d}
\xrightarrow{\partial} C_{d-1}
\; \cdots \;
\xrightarrow{\partial} C_0)
$
is acyclic
since $\fD$ is contractible, i.e., all homology groups are trivial except
for $H_{0}(\mathscr{C})=\bK$.
This gives an  exact sequence
\begin{equation*}
  \label{ex_seq}
0 \rightarrow 
C_{d} \xrightarrow{\partial} C_{d-1}
\xrightarrow{\partial} \ \cdots
\xrightarrow{\partial} C_0 \rightarrow \bK \rightarrow 0
\end{equation*}
of smooth representations of $G^\star$
where $\bK$ is the trivial representation. 

Let $\sigma = ((P_i),\tau)$ be an oriented simplex in $\fD$.
Its stabiliser $\stab_{G^\star}(\sigma)$ is open:
it is either $P'_0$, or its subgroup of index 2,
depending on whether an element of $G^\star$
can reverse the orientation $\tau$ or not.
The one-dimensional space
$\bK [\sigma]$
is a smooth representation of $P'_0$.
Since $P'_0$ is compact, $\bK [\sigma]$
is projective in $\mM (P'_0)$. 
Since $P'_0$ is open, the algebraic induction 
$\bK G^\star \otimes_{\bK P'_0}$
is left adjoint to the restriction functor
$\mM(G^\star)\rightarrow\mM(P'_0)$.
Hence, 
$\bK G^\star \otimes_{\bK P'_0} \bK [\sigma]$
is a projective module in $\mM(G^\star)$. 
Observe that
$$
C_m \cong \oplus_{(P_i)} \bK G^\star \otimes_{\bK P'_0} \bK [((P_i),\tau)]
$$
where the sum is taken over representatives of $G^\star$-orbits
on $\fD_m$.
It follows that
$\mathscr{C}$
is a projective resolution of the trivial representation $\bK$.

Let $V$ be an object in $\mM (G^\star)$.
Tensor product of representations
$\otimes V$ is an exact functor
$\mM (G^\star)\rightarrow \mM (G^\star)$ so that
$$
0 \rightarrow
C_{d} \otimes V
\rightarrow C_{d-1} \otimes V
\rightarrow \cdots
\rightarrow C_0 \otimes V \rightarrow V \rightarrow 0
$$
is an exact sequence. We claim that it is a projective resolution of $V$.
Indeed, 
the functor $\sF=\hom (C_m, \underline{\quad}\ )$ is exact
since $C_m$ is projective. 
The functor
of all linear maps
$\sE=\hom_{\bK} (V, \underline{\quad}\ )$ is also exact.
The composition of two exact functors is exact, so
$\sF\sE=\hom (C_m \otimes V, \underline{\quad}\ )$
is exact
and $C_m \otimes V$ are projective objects.
\end{proof}

\subsection{Localisation}\label{s2.5}
One should put Theorem~\ref{PD_dim}
into a broader perspective of
Schneider-Stuhler Localisation \cite{HR,SS2}.
By localisation we 
understand an equivalence of two categories:
a representation theoretic category
($\mM(G^\star)$ for us)
is equivalent
to (``localised to'') a geometric category.
The key geometric category
is the category $\Ch$ of $G^\star$-equivariant cosheaves
on $\fD$.

{\em A $G^\star$-equivariant cosheaf},
a.k.a. a coefficient system for homology,
is a datum $\mC = (\mC_F, r^{F}_{F'}, g_F)$
where $\mC_F$ is a $\bK$-vector space for each face $F$ of $\fD$,
$r^{F}_{F'}: \mC_F \to \mC_{F'}$
is a linear map
for each pair of faces $F' \subseteq F$,
$g_F: \mC_F \to \mC_{gF}$ is a linear map
for all $g \in G^\star$ and a face $F$
that are subject to the following axioms:
\begin{itemize}
\item[(i)] $r^{F}_{F} = \id_F$ for every face $F$, 
\item[(ii)] $r^{F'}_{F''} \circ r^{F}_{F'} = r^{F}_{F''}$ for faces
  $F'' \subseteq F' \subseteq F$, 
\item[(iii)] $g_{hF} \circ h_F = (gh)_F$ for all $g, h$ and  $F$,
\item[(iv)] 
  $\mC_F$ is a smooth representation of the stabiliser $G^\star_F$
  for all $F$, 
\vspace{4pt}
\item[(v)]
The square 
$
\begin{CD}
\mC_F @>> g_F> \mC_{g F}\\
@VV{r^F_{F'}}V @VV{r^{gF}_{gF'}}V\\
\mC_{F'} @>{g_{F'}}>> \mC_{gF'}
\end{CD}
$ 
\hspace{14pt}
is commutative
for all $g$ and 
$F' \subseteq F$.
\end{itemize}

A morphism of equivariant cosheaves
$\psi : \mC \rightarrow \mE$ 
is a system of linear maps 
$\psi_F : \mC_F \rightarrow \mE_F$,
commuting with actions
and restrictions, i.e, 
the squares 
$$
\begin{CD}
\mC_F @>>\psi_F> \mE_{F}\\
@VV{r^F_{F'}}V @VV{r^F_{F'}}V\\
\mC_{F'} @>{\psi_{F'}}>> \mE_{F'}
\end{CD}
\hspace{34pt}
\mbox{ and }
\hspace{24pt}
\begin{CD}
\mC_F @>>\psi_F> \mE_{F}\\
@VV{g_F}V @VV{g_F}V\\
\mC_{g F} @>{\psi_{g F}}>> \mE_{g F}
\end{CD}
$$ 
are commutative
for all $g$ and
$F' \subseteq F$.

The category of equivariant cosheaves 
$\Ch$
is an abelian category \cite{SS2}: kernels and cokernels
can be computed simplexwise.
There are several functors connecting the key categories
$\mM(G^\star)$ and $\Ch$.
For instance, 
\emph{the trivial cosheaf functor}
$\mathscr{L}$
associates 
a cosheaf $\Vt\in\Ch$
to $(V,\rho)\in\mM (G^\star)$:
$$
\Vt_F=V, \ \
r^F_{F'} = \id_V, \ \
g_F = \rho (g).
$$
In the opposite direction, if $\mC$ 
is a $G^\star$-equivariant cosheaf,
the group $G^\star$ acts on 
the vector space of oriented $i$-chains
(with finite support) 
$C_i(\fD, \mC)$ with coefficients
in $\mC$.
In fact, $G^\star$ acts on the space of more general chains as well
but the finite support ensures that 
$C_i(\fD, \mC)\in \mM(G^\star)$, a functor in the opposite direction!
Furthermore, the chain complex
$$
\mathscr{C} (\mC):
0 \rightarrow C_d(\fD, \mC) \xrightarrow{\partial}
C_{d-1}(\fD, \mC) \xrightarrow{\partial} \cdots
\xrightarrow{\partial} C_0(\fD, \mC)
\rightarrow 0
$$
is a chain complex in $\mM (G^\star)$.
The functor $\mathscr{C}$ allows us to paraphrase Theorem~\ref{PD_dim}:
\begin{cor}
Given a smooth representation $V$, 
  the complex $\mathscr{C} (\Vt)$
  is a projective resolution of $V$.
  \end{cor}

Resolutions of the form $\mathscr{C} (\mC)$
are quite useful. The category $\Ch$ is {\em Noetherian}:
a subobject of a finitely-generated object
is finitely-generated.
Hence, a finitely-generated infinite-dimensional object $V$
admits 
a finitely-generated projective resolution,
yet $\sC (\Vt)$ is not finitely generated. 
We call a finitely-generated projective resolution
of the form $\mathscr{C} (\mC)$
{\em a Schneider-Stuhler resolution}. Do they exist (cf. Section~\ref{s3.8})?

It may be possible to construct them using systems of subgroups
or system of idempotents \cite{HR,MeSo}. In fact, 
\cite{HR} contains a positive answer to existence of Schneider-Stuhler
resolutions modulo a (yet open) conjecture on homology
of a CAT(0)-complex. To satisfy the reader's curiosity we state this conjecture
in full in Section~\ref{s3.9}.

Let us turn our attention to the localisation.
We have the trivial cosheaf and the 0-th homology functors
going between $\mM(G^\star)$ and $\Ch$:
$$
\sL((V,\rho)) = \Vt, \ \ 
\sH(\sC) = H_0 (\fD, \sC). 
$$
Let
$\Sigma \subset \mbox{Mor}(\Ch)$ 
be the class of morphisms $\psi$ such that $\sH(\psi)$ is an isomorphism.
Consider the category of left fractions
$\Ch[\Sigma^{-1}]$
and the fraction functor 
$\sQ_{\Sigma}: \Ch\rightarrow \Ch[\Sigma^{-1}]$.
Note that while these fractions always exist, 
$\Sigma$ 
needs to satisfy
{\em the left Ore condition}
(a.k.a. admit a calculus of left fractions)
for these objects to be malleable \cite{GZ}.
The 0-th cohomology functor extends
to
a functor from the category of left fractions
$
\sH[\Sigma^{-1}]: \Ch[\Sigma^{-1}] \rightarrow \mM(G^\star)
$. 
We are ready for  the localisation theorem,
a generalisation of Schneider-Stuhler Localisation \cite{SS2}:
\begin{thm} \cite{HR}
\label{local}
Under the notations established above, 
the following statements hold:
\begin{itemize}
\item[(i)]   The class $\Sigma$
  satisfies the left Ore condition.
\item[(ii)]   
The functor $\sH[\Sigma^{-1}]: \Ch[\Sigma^{-1}] \rightarrow \mM(G^\star)$
is an equivalence of categories. 
\item[(iii)] $\sQ_{\Sigma}\circ \sL$ is a quasi-inverse of
  $\sH[\Sigma^{-1}]$.
\end{itemize}
\end{thm}


\section{Questions}
\subsection{Isomorphism Problem}
\label{s3.1}
Find necessary and sufficient conditions on realisations
$\mR$ and $\mS$ for
$\widehat{L}_\mR$
and
$\widehat{L}_\mS$
to be equivalent as functors to graded Lie algebras.

\subsection{Existence of Restricted Structure}
\label{s3.2}
Suppose $\bF$ is a field of positive characteristic.
Find necessary and sufficient conditions on realisation for
$\widehat{L}_\mR (\bF)$ to admit a structure of restricted Lie algebra.
In particular, 
if $\mA$ is a generalised Cartan matrix of general type
and $p \leq \max_{i\neq j} (-A_{i,j})$,
could $\widehat{L}_\mD (\bF)$ be restricted?

\subsection{Humphreys-Verma Conjecture}
\label{s3.3}
Consider ``natural'' restricted $\fg_\bF$-modules,
e.g., irreducible, projective, injective.
Do they admit an action of $G_\mD (\bF)$
such that for each real root $\alpha$
the differential of the $U_\alpha$-action
is the $\fg_{\bF \; \alpha}$-action?

\subsection{Theory of Over-restricted Representations}
\label{s3.4}
Investigate algebraic
properties of the over-restricted enveloping algebra
$U (\fg_\bF) /(x^p-x^{[p]},e_\alpha^{\lfloor (p+1)/2\rfloor})$
and its representations. 



\subsection{Congruence Kernel}
\label{s3.5}
Develop techniques for computing $C(G^{lpp})$.
Find necessary and sufficient conditions for $C(G^{lpp})$
to be trivial (central, finitely pro-$p$-generated, etc.).

\subsection{Lattices in Locally Pro-$p$-complete Kac-Moody Groups}
\label{s3.6}
Find minimal covolume of lattices (uniform and overall)
in $G^{lpp}$.

\subsection{Completions}
\label{s3.7}
Investigate the completions.
Find necessary and sufficient conditions the following completions to be equal
$G^{b}\stackrel{?}{=} G^{lpp}$,
$G^{+}\stackrel{?}{=} G^{ma+}$,
$G^{c\lambda}\stackrel{?}{=} G^{crr}$.

\subsection{Schneider-Stuhler Resolution}
\label{s3.8}
Does a Schneider-Stuhler resolution exist for any finitely-generated
object
$V\in \mM(G^\star)$?
What about irreducible objects? 
More precisely, does there exist a family of functors
$\mathscr{T}_k: \mM (G^\star)\rightarrow \Ch$, indexed by natural numbers,
such that for each irreducible $L\in\mM(G^\star)$
there exists $N\in\bN$ such that
$\mathscr{C} (\mathscr{T}_k (L))$
is a Schneider-Stuhler resolution of $L$ for all $k>N$.

\subsection{Homology of CAT(0)-Complex}
\label{s3.9}
Let $\fX$ be a CAT(0)-simplicial complex, $A$ an abelian group.
Suppose we have an idempotent operator $\Lambda_x :A \rightarrow A$
for each vertex $x$ of $\fX$. We call this system of idempotents
{\em geodesic} if the following conditions hold:
\begin{itemize}
\item[(i)] $\Lambda_x \Lambda_y = \Lambda_y \Lambda_x$ if $x$ and $y$ are adjacent,
\item[(ii)] $\Lambda_x \Lambda_z \Lambda_y = \Lambda_x \Lambda_y$
  and
  $\Lambda_x \Lambda_z  = \Lambda_z \Lambda_x$
  if 
  $z$ is any vertex of the first simplex
  along the geodesic $[x,y]$ for all vertices $x$ and $y$.
  \end{itemize}
Such geodesic system gives a cosheaf
$\Al$ where $\Al_F$ is the image of the product $\prod_x \Lambda_x$
taken over all faces of $F$ and $r^F_{F'}$ are natural inclusions.

Is it true that $H_{m}(\mX, \Al)=0$ for all $m>0$?

A positive answer to this question for Bruhat-Tits buildings can be
obtained by the methods of Meyer and Solleveld \cite{MeSo}.
%

\end{document}